\documentclass{article}
\usepackage[utf8]{inputenc}
\usepackage{amsmath}
\usepackage{amsthm}
\usepackage{amssymb}
\usepackage{enumerate}
\usepackage{hyperref}
\usepackage{aliascnt}
\usepackage{tikz}
\usepackage[affil-it]{authblk}

\newtheorem{THM}{Theorem}[section]
\newtheorem{LEM}[THM]{Lemma}
\newtheorem{COR}[THM]{Corollary}

\makeatletter
\newtheorem*{rep@theorem}{\rep@title}
\newcommand{\newreptheorem}[2]{%
\newenvironment{rep#1}[1]{%
 \def\rep@title{#2 \ref{##1}}%
 \begin{rep@theorem}}%
 {\end{rep@theorem}}}
\makeatother

\newreptheorem{theorem}{Theorem}
\newreptheorem{lemma}{Lemma}

\newenvironment{subproof}[1][\proofname]{%
  \begin{proof}[#1]%
}{%
  \end{proof}%
}

\newtheorem{clm}{Claim}

\newcommand{\tuple}[1]{\ensuremath{\left \langle #1 \right \rangle }}

\begin{document}

\title{Tree-chromatic number is not equal to path-chromatic number\footnote{This research was supported by the European Research Council under the European Unions Seventh Framework
Programme (FP7/2007-2013)/ERC Grant Agreement no. 279558.}}

\author{Tony Huynh%
  \thanks{Email: \texttt{tony.bourbaki@gmail.com}}}
\affil{Department of Mathematics \\
  Universit\'e Libre de Bruxelles\\
  Boulevard du Triomphe, B-1050 \\
  Brussels, Belgium}

\author{Ringi Kim%
  \thanks{Email: \texttt{ringi.kim@uwaterloo.ca}}}
\affil{Department of Combinatorics and Optimization \\
University of Waterloo \\
Waterloo, Ontario N2L 3G1, Canada}

\date{Dated: \today}

\maketitle

\begin{abstract}
For a graph $G$ and a tree-decomposition $(T, \mathcal{B})$ of $G$, the {\em chromatic number} of $(T, \mathcal{B})$ is the maximum of $\chi(G[B])$, taken over all bags $B \in \mathcal{B}$.  The \emph{tree-chromatic number} of $G$ is the minimum chromatic number of all tree-decompositions $(T, \mathcal{B})$ of $G$. The \emph{path-chromatic number} of $G$ is defined analogously. In this paper, we introduce an operation that always increases the path-chromatic number of a graph.  As an easy corollary of our construction,  we obtain an infinite family of graphs whose path-chromatic number and tree-chromatic number are different.  This settles a question of Seymour \cite{Seymour16}.  Our results also imply that the path-chromatic numbers of the Mycielski graphs are unbounded. 
\end{abstract}

\section{Introduction}
A \emph{tree-decomposition} of a graph $G$ is a pair $(T, \mathcal{B})$
where $T$ is a tree and $\mathcal{B}:=\{B_t \mid t \in V(T)\}$ is a collection of subsets of vertices of $G$ satisfying:
\begin{itemize}
\item $V(G)= \bigcup_{t \in V(T)} B_t$, 
\item for each $uv \in E(G)$, there exists $t \in V(T)$ such that $u,v \in B_t$, and
\item for each $v \in V(G)$, the set of all $w \in V(T)$ such that $v \in B_w$ induces a connected subtree of $T$.  
\end{itemize}
We call each member of $\mathcal{B}$ a {\em bag}.
If $T$ is a path, then we say  a tree-decomposition $(T,\mathcal{B})$ of $G$ is  a {\em path-decomposition} of $G$. Since a path can be written as a sequence of vertices, we think of a path-decomposition of $G$ as a sequence of sets of vertices $B_1,B_2,\ldots,B_s$ such that 
\begin{itemize}
\item $V(G)= \bigcup_{1\le t \le s} B_t$, 
\item for each $uv \in E(G)$, there exists $1\le t \le s$ such that $u,v \in B_t$, and
\item for each $v \in V(G)$, the sets $B_i$ containing $v$ are consecutive in the sequence.
\end{itemize}

For a tree-decomposition $(T,\mathcal{B})$ of $G$, the \emph{chromatic number} of $(T, \mathcal{B})$ is $\max \{\chi(G[B_t]) \mid t \in V(T)\}$.  The \emph{tree-chromatic number} of $G$, denoted $\chi_T(G)$, is the minimum chromatic number taken over all tree-decompositions of $G$.  The \emph{path-chromatic number} of $G$, denoted $\chi_P(G)$, is defined analogously, where we insist that $T$ is a path instead of an arbitrary tree.  Both these parameters were introduced by Seymour \cite{Seymour16}.   Evidently, $\chi_T(G) \leq \chi_P(G) \leq \chi(G)$ for all graphs $G$.

The \emph{closed neighborhood} of a set of vertices $U$, denoted $N[U]$, is the set of vertices with a neighbor in $U$, together with $U$ itself. For every enumeration $\sigma=v_1,\ldots,v_n$ of the vertices of a  graph $G$, we denote by $P_{\sigma}^G$ the sequence $X_1, \ldots, X_n$ of sets of vertices of $G$ such that 
\[
X_{\ell}=N[ \{v_1, \dots, v_{\ell}\}] \setminus \{v_1,\ldots,v_{\ell-1}\}.
\] 

Observe that every vertex $v_i$ of $G$ belongs to $X_i$, and for $v_iv_j \in E(G)$ with $i<j$, both $v_i$ and $v_j$ belong to $X_i$. Furthermore, for $v_i \in V(G)$, if $m$ is the first index such that $v_i \in N[\{v_m\}]$, then $v_i \in X_{j}$ if and only if $ m\le j \le i$. So,  $P_{\sigma}^G$ is indeed a path decomposition of $G$. Let $\chi(P_{\sigma}^G)$ be the chromatic number of $P_{\sigma}^G$.

The following shows that for every graph $G$, there is an enumeration $\sigma$ of $V(G)$ such that $\chi(P_{\sigma}^G)=\chi_P(G)$.

\begin{LEM}\label{enumeration}
If $G$ has path-chromatic number $k$, then there is some enumeration $\sigma$ of $V(G)$ such that $P_{\sigma}^G$ has chromatic number $k$.
\end{LEM} 
We prove this later in this section. Furthermore, the obvious modification of a standard dynamic programming algorithm (see Section 3 of \cite{SV2009}) yields a $O(n4^n)$-time algorithm to test if $G$ has path-chromatic number at most $k$.


We write  $[n]$ for $\{1,2,\ldots,n\}$. 
For a graph $G$ with vertex set $V(G)$, 
let $R_m(G)$ be the graph with vertex set $\{(i,v) \mid i \in [m], v \in V(G) \cup \{v_0\}\}$ where $v_0 \not \in V(G)$,  such that two distinct $(i,v)$ and $(i',v')$ are adjacent  if and only if one of the following holds:
\begin{itemize}
\item $i=i'$ and exactly one of $v$ or $v'$ is $v_0$, or
\item $i\neq i'$, $v, v' \in V(G)$ and $v v' \in E(G)$.
\end{itemize}

For a subset of vertices $S$, we let $\tuple{S}$ denote the subgraph induced by $S$ (the underlying graph will always be clear).  We also abbreviate $\chi(\tuple{S})$ by $\chi(S)$. The main theorems of this paper are the following.
For an enumeration $\sigma=v_1,\ldots,v_n$ of $V(G)$ with $P_{\sigma}^G=X_1,X_2,\ldots,X_n$, we say $\sigma$ is {\em special} if 
\begin{itemize}
\item $\chi(P_{\sigma}^G)=\chi_P(G)$ and
\item for every $1\le i \le n$ with $\chi(X_i)=\chi_P(G)$, $v_i$ has no neighbor in $\{v_1,\ldots,v_{i-1}\}$.
\end{itemize}

\begin{THM}\label{mainthm}
Let $n$ and $k$ be positive integers, with $k \geq 2$. For every integer $m \geq n+k+2$ and every graph $G$ with $\chi_p(G)=k$ and $|V(G)|=n$, the path-chromatic number of $R_m(G)$ is $k$ if there is a special enumeration of $V(G)$. Otherwise, the path-chromatic number of $R_m(G)$ is $k+1$.

\end{THM}

Theorem~\ref{mainthm} does not guarantee that applying $R_m$ always increases path-chromatic number.  On the other hand, our second theorem shows that applying $R_m$ \emph{twice} always increases path-chromatic number. 

\begin{THM}\label{mainsec}
Let $G$ be a graph with $\chi_P(G)=k$ and $|V(G)|=n$. For all integers $\ell$ and $m$ such that $m \ge n+k+2$ and $\ell \ge m (n+1)+k+3$, the path-chromatic number of $R_{\ell}(R_{m}(G))$ is strictly larger than $k$.
\end{THM}

Theorem \ref{mainthm} easily implies the following corollary. 

\begin{COR} \label{notequal}
For every positive integer $k$, there is an infinite family of $k$-connected graphs $G$ for which $\chi_T(G) \neq \chi_P(G)$.
\end{COR}

These are the first known examples of graphs with differing tree-chromatic and path-chromatic numbers, which settles a question of Seymour \cite{Seymour16}.  Seymour also suspects that there is no function $f: \mathbb{N} \to \mathbb{N}$ for which $\chi_P(G) \leq f(\chi_T(G))$ for all graphs $G$, but
unfortunately our results are not strong enough to derive this stronger conclusion.  

Our results also imply that the family of Mycielski graphs have unbounded path-chromatic numbers.
For $k\geq 2$, the \emph{$k$-Mycielski graph $M_k$}, is the graph with $3 \cdot 2^{k-2}-1$ vertices constructed recursively in the following way. $M_2$ is a single edge and $M_k$ is obtained from $M_{k-1}$ by adding $3\cdot 2^{k-3}$ vertices $w,u_1,u_2,\ldots,u_{3\cdot 2^{k-3}-1}$ and adding edges $wu_i$ for all $i$ and $u_iv_j$ for all $i\neq j$ such that $v_iv_j \in E(M_{k-1})$ where $v_1,v_2,\ldots,v_{3\cdot 2^{k-3}-1}$ are the vertices of $M_{k-1}$. Here we say $u_i$ \emph{corresponds} to $v_i$. It is easy to show (see \cite{Mycielski55}) that for all $k \geq 2$,  $M_k$ is triangle-free and $\chi(M_k)=k$.

\begin{COR} \label{unbounded}
For every positive integer $c$, there exists a positive integer $n(c)$ such that the $n(c)$-Mycielski graph has path-chromatic number larger than $c$.
\end{COR}  

We prove Corollary~\ref{notequal} and Corollary~\ref{unbounded} in Section~\ref{sec:corollaries}. 
We finish this section by proving Lemma~\ref{enumeration}.

\begin{proof}[Proof of Lemma~\ref{enumeration}.]
For every path-decomposition $(P,\mathcal{B})=B_1,B_2,\ldots,B_s$ of $G$, we prove that there exists an enumeration $\sigma$ of $V(G)$ such that the chromatic number of $P_{\sigma}^G$ is at most that of $(P,\mathcal{B})$.  Let $\sigma=v_1,v_2,\ldots,v_n$ be an enumeration of $V(G)$ such that for all $u,v \in V(G)$, if the last bag of $(P,\mathcal{B})$ containing $u$ comes before the last bag of $(P,\mathcal{B})$ containing $v$ then $u$ comes before $v$ in $\sigma$.  It is easy to show that such an enumeration always exists. 
Let $P^G_{\sigma}=X_1,X_2,\ldots,X_n$ and for $1\le i \le n$, let $B_{\ell(i)}$ be the last bag of $(P,\mathcal{B})$ containing $v_i$. 
It is enough to prove that for $1\le i \le n$, $B_{\ell(i)}$ contains $X_i$.

Suppose $v_j \in X_i \setminus B_{\ell(i)}$. 
Obviously $i\neq j$, and since $v_j \in X_i$, we obtain $i<j$. 
Let $B_{f(j)}$ be the first bag of $(P,\mathcal{B})$ containing $v_j$.
Since the bags containing $v_j$ are consecutive in $(P,\mathcal{B})$, $v_j \not \in B_{\ell(i)}$ and $\ell(i) \le \ell(j)$, we obtain that $\ell(i)<f(j)$.
Let $v_k$ be a neighbour of $v_j$ with $k \le i$. Such a $v_k$ exists since $v_j \in X_i$. 
Then, $\ell(k) \le \ell(i) $ since $k\le i$, so $\ell(k) <f(j)$. 
Therefore, there is no bag of $(P,\mathcal{B})$ containing both $v_k$ and $v_j$ because the last bag containing $v_k$ comes before the first bag containing $v_j$. But this is a contradiction since $v_kv_j \in E(G)$.
Thus, $X_i \subseteq B_{\ell(i)}$ as claimed, and we deduce that the chromatic number of $P_{\sigma}^G$ is at most that of $(P,\mathcal{B})$.
\end{proof}

\section{Deriving the Corollaries} \label{sec:corollaries}

Assuming Theorems \ref{mainthm} and \ref{mainsec}, it is straightforward to derive Corollaries \ref{notequal} and \ref{unbounded}, which we do in this section. 
Let $C_n$ denote the $n$-cycle.

\begin{LEM}\label{cycle}
For all odd integers $n \ge 5$ and all integers $m \geq n+4$, the path-chromatic number of $R_m(C_n)$ is 3.  
\end{LEM}

\begin{proof}
Evidently, $\chi_P(C_n)=2$. Hence, by Theorem~\ref{mainthm}, it is enough to show that every enumeration $\sigma = v_1, \dots, v_n$ of $V(C_n)$ is not special.  Let $P_{\sigma}^G=X_1,X_2,\ldots,X_n$.

Let $(L,M,R)$ be the partition of $V(C_n)$ such that for every $v\in V(C_n)$,
\begin{itemize}
\item $v \in L$ if both neighbors of $v$ come before $v$ in $\sigma$,
\item $v\in R$ if both neighbors of $v$ come after $v$ in $\sigma$,
\item $v \in M$ otherwise.
\end{itemize}

Suppose $M$ is not empty and let $v_{\ell}$ be a vertex of $M$. Obviously, the chromatic number of $\tuple{X_{\ell}}$ is at least 2 because it contains both $v_{\ell}$ and a neighbor of $v_{\ell}$. However, $v_{\ell}$  has a neighbor appearing before $v_{\ell}$ in $\sigma$, so $\sigma$ is not special.  
So, we may assume $M$ is empty. 
Since $L$ and $R$ are both stable sets, it follows that $C_n$ is 2-colorable, a contradiction.
This completes the proof.
\end{proof}

On the other hand, we also have the following easy lemma. 

\begin{LEM} \label{c5treechi}
For all integers $n \ge 4$ and all positive integers $m$, $R_m(C_n)$ has tree-chromatic number $2$.  
\end{LEM}

\begin{proof}
It clearly suffices to show that $R_m(C_n)$ has tree-chromatic number at most $2$.  Let $V(C_n)=\{v_1, \dots, v_n\}$ with $v_j$ adjacent to $v_{j'}$ if and only if $|j-j'| \in \{1,n-1\}$. Let the vertex set of $R_m(C_n)$ be $\{(i,v_j)\mid i \in [m], j\in \{0\} \cup [n]\}$. We now describe a tree-decomposition $(T, \mathcal{B})$ of $R_m(C_n)$. Let $T$ be a star with a center vertex $c$ and $m$ leaves $\ell(1),\ldots,\ell(m)$. Let 
\begin{itemize}
\item 
$B_c=\{(i,v_j) \mid i \in [m], j\in \{2,3,\ldots,n\}\}$, 
\item
$B_{\ell(s)}=\{(s,v_j) \mid j \in \{0,1,2,\ldots,n\}\} \cup \{(i,v_j)\mid i \in [m], j\in \{2,n\}\}$.

\end{itemize}
We claim that $(T,\mathcal{B})$ is a tree-decomposition of $R_m(C_n)$ where $\mathcal{B}=\{ B_t \mid t \in V(T)\}$. 
For $i \in [m]$ and $v_j \in V(C_n)\cup \{v_0\}$, the vertex $(i,v_j)$ of $R_m(C_n)$ belongs to $B_{\ell(i)}$.
If two distinct vertices $(i,v_j)$ and $(i',v_{j'})$ of $R_m(C_n)$ are adjacent, then either $i=i'$ and one of $v_j$ and $v_{j'}$ is $v_0$ or $i\neq i'$, $j,j'\in [n]$ and $|j-j'| \in \{1,n-1\}$. 
If the first case holds, then both vertices belong to $B_{\ell(i)}$. If the second case holds, 
then if either $v_j=v_1$ or $v_{j'}=v_1$ then both vertices belong to $B_{\ell(i)}$, and if neither $v_j$ nor $v_{j'}$ is $v_1$, then both belong to $B_c$. 
Lastly, for $(i,v_j) \in R_m(C_n)$, if $v_j \notin \{v_0,v_1\}$ then $(i,v_j)$ belongs to $B_c$, so $\{t \mid (i,v_j) \in B_t\}$ automatically induces a subtree in $T$. 
If $v_j$ is either $v_0$ or $v_1$, then only $B_{\ell(i)}$ contains $(i,v_j)$. Hence, $(T,\mathcal{B})$ is a tree-decomposition, as claimed.  

The set of vertices $(i,v_j)$ of $B_c$ with even $j$ (or odd $j$)  is independent. Hence, $\chi(B_c)$ is at most $2$. Moreover, for $i\in [m]$, both of $\{(i,v)\mid  v \in V(C_n)\}$ (note $v_0 \notin V(C_n)$) and $B_{\ell(i)}\setminus \{(i,v)\mid v\in V(C_n)\}$ are independent, so $\chi(B_{\ell(i)})$ is at most $2$. We conclude that $(T,\mathcal{B})$ has chromatic number at most $2$. This completes the proof.
\end{proof}

\begin{proof}[Proof of Corollary~\ref{notequal}]
For every odd integer $n \ge 5$ and every integer $m\ge n+4$,  
Lemma~\ref{cycle} and Lemma~\ref{c5treechi} show that the tree-chromatic number and path-chromatic number of $R_m(C_{n})$ are different.
To complete the proof, we prove that $R_m(C_{n})$ is $k$-connected for every $n \ge k$ and $m \geq n+4$.   
We prove that for every set $U$ of vertices of $R_m(C_{n})$ of size at most $k-1$, 
$R_m(C_{n}) - U$ is connected. 
Again, let $V(C_n)=\{v_1,v_2,\ldots,v_n\}$ with $v_j$ adjacent to $v_{j'}$ if $|j-j'| \in \{1,n-1\}$ and $V(R_m(C_n))=\{(i,v_j)\mid i \in [m], j \in \{0\} \cup [n]\}$.

Since $m>|U|$, there exists $i^* \in [m]$ such that no vertex in $\{(i^*,v_j)\mid j \in \{0\}\cup [n]\}$ belongs to $U$.
Without loss of generality, $i^*=1$. 
We claim that for every vertex $(i,v_j)$ of $R_m(C_{n}) - U$, there is a path from $(1,v_0)$ to $(i,v_j)$. 
We may assume that $(i,v_j) \neq (1,v_0)$. 
If $i=1$, then $(1,v_0),(1,v_j)$ is a path. Hence, we may assume that $i\neq 1$. If $v_j\neq v_0$, then $(1,v_0), (1,v_{j+1}), (i,v_j)$ is a path, where $(1,v_{n+1})=(1,v_1)$.
If $v_j=v_0$, there exists $j' \in [n]$ such that $(i,v_{j'}) \notin U$ since $n > |U|$. Then $(1,v_0),(1,v_{j'+1}),(i,v_{j'}), (i,v_0)$ is a path. Therefore, $R_m(C_n) - U$ is connected. This completes the proof.
\end{proof}



Recall that $M_k$ denotes the $k$-Mycielski graph.  

\begin{LEM}\label{Mycielskian}
For all positive integers $n,m$ and all integers $r\ge m+n$,  $M_r$ contains $R_m(M_n)$ as an induced subgraph.
\end{LEM}

\begin{proof}
Take a sequence $G_n, G_{n+1},\cdots,G_r$ of induced subgraphs of $M_r$ where $G_i$ is isomorphic to $M_i$ and $G_i$ is an induced subgraph of $G_{i+1}$ for $i=n,n+1,\ldots,r-1$. Let $V(G_n)= \{v^n_1,v^n_2,\ldots,v^n_{3\cdot 2^{n-2}-1}\}$, and  for $s > n$, let $V(G_s) \setminus V(G_{s-1})=\{v^s_0,v^s_1,\ldots,v^s_{3\cdot 2^{s-3}-1}\}$ where $v^s_0$ is complete to the other vertices in this set and $v^s_{i}$ corresponds to $v^n_i$ for $ 1\le i \le 3\cdot 2^{n-2}-1$. Then, the graph induced by $\{v^x_y \mid n+1\le x \le m+n , 0\le y \le 3 \cdot 2^{n-2}-1\}$ is isomorphic to $R_m(M_n)$.
\end{proof}

Lemma \ref{Mycielskian} and Theorem \ref{mainsec} together imply Corollary \ref{unbounded}.  Thus, it only remains to prove Theorems \ref{mainthm} and  \ref{mainsec}, which we do in the remaining section.

\section{Proofs of Theorems \ref{mainthm} and \ref{mainsec}}

In this section, we prove Theorem \ref{mainthm} and Theorem \ref{mainsec}.  
Throughout this section, $G$ is a graph with $n$ vertices and $R_m(G)$ has vertex set $\{(i,v) \mid  i \in [m],  v\in V(G)\cup\{v_0\}\}$.  
For $I \subseteq [m]$ and $U \subseteq V(G)\cup\{v_0\}$,  we set $[I,U]=\{(i,v) \mid i \in I, v\in U\}$.
We start with the following lemmas.

\begin{LEM}\label{isomorphic}
For $I \subseteq [m]$ and $U\subseteq V(G)$, suppose $|I| \ge \chi(U)$. Then there exists a map $f:U \to [I,U]$ such that 
\begin{itemize}
\item for every $v \in U$, $f(v)$ belongs to $[I,\{v\}]$, and
\item $f$ is an isomorphism from $\tuple{U}$ to $\tuple{f(U)}$.
\end{itemize}
Furthermore,  for all $i^* \in [m]\setminus I$ and all $v^* \in V(G)\setminus U$, $\tuple{[I,U]\cup \{(i^*,v^*)\}}$ contains an isomorphic copy of $\tuple{U\cup \{v^*\}}$ as an induced subgraph. 
\end{LEM}

\begin{proof}
Let $\chi(U)=c$. Let $\mathcal{U}=(U_1,U_2,\ldots,U_c)$ be a partition of $U$ into independent sets of $G$.
Take $c$ distinct elements from $I$, say $i_1,i_2,\ldots,i_c$, and for $v\in U$, let $f(v)=(i_s,v)$ if $v\in U_s$.
We claim that $f$ is an isomorphism from $\tuple{U}$ to $\tuple{f(U)}$.

Let $v$ and  $v'$ be distinct vertices in $U$.  If $v$ and $v'$ are adjacent, they are contained in distinct classes of $\mathcal{U}$, so $f(v)$ and $f(v')$ are adjacent by the definition of $R_m(G)$. If $v$ and $v'$ are non-adjacent, there are no edges between $[I,\{v\}]$ and $[I,\{v'\}]$. Hence, $f(v)$ and $f(v')$ are non-adjacent.
Thus, $f$ is an isomorphism from $\tuple{U}$ to $\tuple{f(U)}$.

For the last part, let $i^* \in [m]\setminus I$ and $v^* \in V(G)\setminus U$. Let $f^*$ be the map obtained from $f$ by adding $f^*(v^*)=(i^*,v^*)$. Since $i^*\not \in I$, it easily follows that $f^*$ is an isomorphism from $\tuple{U\cup \{v^*\}}$ to $\tuple{f^*(U\cup \{v^*\})}$.
This completes the proof.
\end{proof}

When considering $k$-colorings of a graph, we always use $[k]$ for the set of colors.  

\begin{LEM}\label{coloring}
For $I \subseteq [m]$ and $U \subseteq V(G)$, let $\chi( U)=c$. If  $|I|\geq c$, the chromatic number of $\tuple{[I,U]}$ is $c$. Moreover, if  $|I| > c$, then for every $c$-coloring $C$ of $\tuple{[I,U]}$ and every $i \in I$,  $[\{i\},U]$ uses all $c$ colors of $C$. In other words, $C([\{i\},U])=[c]$ for every $i\in I$.
\end{LEM}

\begin{proof}
Let $(U_1,U_2,\ldots,U_c)$ be a partition of $U$ into independent sets of $G$.
Then, $([I,U_1],[I,U_2],\ldots,[I,U_c])$ is a partition of $[I,U]$ and each set is independent in $\tuple{[I,U]}$. Hence, the chromatic number of $\tuple{[I,U]}$ is at most $c$. 
On the other hand, $\chi([I,U]) \geq c$ follows from Lemma \ref{isomorphic}.
Thus, the chromatic number of $\tuple{[I,U]}$ is $c$.

For the second part, let $C:[I,U] \to [c]$ be a $c$-coloring of $\tuple{[I,U]}$. Fix $i \in I$. Since $|I\setminus i|$ is still greater than or equal to $c$, we can apply Lemma~\ref{isomorphic} to $[I\setminus i,U]$. Let $f$ be a map from $U$ to $[I\setminus i,U]$ as in the statement of Lemma~\ref{isomorphic}. 
Let $F=f(U)$, and $C_F$ be the restriction of $C$ on $F$. 
As $\tuple{f(U)}$ is not $(c-1)$-colorable, for each color $\alpha \in [c]$, there must be a vertex $v_{\ell_{\alpha}} \in U$ such that  $f(v_{\ell_{\alpha}}) \in C_F^{-1}(\alpha)$ and $f(v_{\ell_{\alpha}})$ has a neighbor in $C_F^{-1}(\beta)$ for every $\beta \in [c]\setminus \alpha$.  
Then, $(i,v_{\ell_{\alpha}})$ also has a neighbor in $C_F^{-1}(\beta)$ for every $\beta \in [c]\setminus \alpha$, so $C((i,v_{\ell_{\alpha}}))$ is $\alpha$.  
Hence, $[\{i\},U]$ sees all colors, which proves the second part.
\end{proof}

\begin{LEM}\label{special}
For a graph $G$ with path-chromatic number $k \geq 2$, let $\sigma=v_1,v_2,\ldots,v_n$ be a special vertex enumeration of $G$. 
Let $P_{\sigma}^{G}=X_1,X_2,\ldots,X_n$.
For $j\in [n]$, if $\chi(X_{j})=k$ then $\chi(X_j \setminus v_j) = k-1$.
\end{LEM}

\begin{proof}
It is obvious that $\chi(X_j \setminus v_j) \ge k-1$.
We may assume that $X_j \setminus v_j \neq \emptyset$.
Let $j'$ be the smallest index such that $v_{j'} \in X_j\setminus v_j$. Note that $j'>j$ and since $v_{j'}$ is contained in $X_{j}$, it has a neighbor in $\{v_1,\ldots,v_{j}\}$. Hence, by the definition of a special vertex enumeration, $\chi(X_{j'}) \leq k-1$. However, by the choice of $j'$,  $X_{j} \setminus v_j$ is a subset of $X_{j'}$.
Thus, $\chi(  X_{j} \setminus v_j)\le k-1$, as required.
\end{proof}

For an enumeration $\sigma$ of vertices and a vertex $v$, let $\sigma(<v)$ denote the set of vertices which come before $v$ in $\sigma$ and $\sigma(\le v)=\sigma(<v)\cup \{v\}$.

\begin{LEM}\label{bound}
Let $m\ge n+1$ and $\mu=(i_1,v_{j_1}),(i_2,v_{j_2}),\ldots, (i_{m(n+1)}, v_{j_{m(n+1)}})$ be an enumeration of the vertices of $R_m(G)$. Let $k$ be the chromatic number of $P_{\mu}^{R_m(G)}$.  
For each $v\in V(G)$, let $t(v)$ be the vertex in $[[m],\{v\}]$ which comes first in $\mu$. Suppose that for all $1\le j < j' \le n$, 
$t(v_j)$ comes before $t(v_{j'})$ in $\mu$ and let
 $\sigma=v_1,v_2,\ldots,v_n$ be the corresponding enumeration of $V(G)$. Let $P_{\sigma}^G=X_1,X_2,\ldots,X_n$.  Then,
\begin{itemize}
\item[(1)] the chromatic number of $P_{\sigma}^G$ is at most $k$, and
\item[(2)] if $\chi( X_{\ell})=k$ for some $\ell \in [n]$, then $\mu( \leq t(v_{\ell}))$ contains at most $k$ vertices in $[[m],\{v_0\}]$.
\end{itemize}
\end{LEM}
\begin{proof}
For all $v \in V(G)$, let $f(v)\in [m]$ be such that $t(v)=(f(v),v)$. 
Let $P_{\mu}^{R_m(G)} = Y_{(i_1,v_{j_1})},Y_{(i_2,v_{j_2})},\ldots, Y_{(i_{m(n+1)}, v_{j_{m(n+1)}})}$. 
For the first statement, it suffices to show that for all $\ell \in [n]$, $\tuple{Y_{(f(\ell),v_{\ell})}}$ contains $\tuple{X_{\ell}}$ as an induced subgraph. 
Let $I=[m] \setminus \{f(v_1),\ldots, f(v_\ell)\}$.
Then, $|I| \ge m-\ell \ge n+1-\ell =1+(n-\ell) > |X_{\ell}\setminus v_\ell| \ge \chi ( X_{\ell}\setminus v_\ell)$. 
Moreover, $f(v_\ell) \not \in I$ and $v_{\ell} \not \in X_{\ell}\setminus v_{\ell}$. By Lemma~\ref{isomorphic}, $\tuple{[I, X_\ell \setminus v_\ell] \cup \{(f(v_\ell), v_\ell)\}}$ contains $\tuple{X_{\ell}}$ as an induced subgraph.
Since $t(v_j)$ comes before $t(v_{j'})$ in $\mu$ for all $1\le j < j' \le n$, it follows that $Y_{(f(\ell),v_{\ell})}$ contains $[I, X_\ell \setminus v_\ell] \cup \{(f(v_\ell), v_\ell)\}$, as required.   


For the second statement, 
suppose $\mu( \leq t(v_{\ell}))$ has exactly $r$ vertices in $[[m],\{v_0\}]$, with $r \geq k+1$.  
By relabeling, we may assume that $(i,v_0)$ is in $\mu( \leq t(v_{\ell}))$ for all $i \in [r]$ and that $(r,v_0)$ appears last in $\mu$ among them. 
Observe that $Y_{(r,v_0)}$ contains $\{(r,v_0)\} \cup [[r],X_{\ell}]$. 
Let $C$ be  a $k$-coloring of $Y_{(r,v_0)}$.
By Lemma \ref{coloring}, since $r > \chi(X_{\ell})$, 
for every $k$-coloring of
$\tuple{[[r],X_{\ell}]}$, $[\{r\},X_{\ell}]$ sees all $k$ colors.
Hence $C([\{r\},X_{\ell}])=[k]$.
But then there is no available color for $(r, v_0)$, which yields a contradiction.
This completes the proof.
\end{proof}

We are now ready to prove Theorem \ref{mainthm}, which we restate for the reader's convenience.  

\begin{reptheorem}{mainthm}
Let $n$ and $k$ be positive integers, with $k \geq 2$. For every integer $m \geq n+k+2$ and every graph $G$ with $\chi_p(G)=k$ and $|V(G)|=n$, the path-chromatic number of $R_m(G)$ is $k$ if there is a special enumeration of $V(G)$. Otherwise, the path-chromatic number of $R_m(G)$ is $k+1$.
\end{reptheorem}

\begin{proof}
By Lemma \ref{isomorphic}, $R_m(G)$ contains $G$ as an induced subgraph, so $\chi_P(R_m(G)) \ge \chi_P(G) =k$. We break the proof up into a series of claims.  

\begin{clm}
If the path-chromatic number of $R_m(G)$ is $k$, then there exists a special enumeration of the vertices of $G$.
\end{clm}

\begin{subproof}[Subproof] 
Let $\mu=(i_1,v_{j_1}),(i_2,v_{j_2}),\ldots, (i_{m(n+1)}, v_{j_{m(n+1)}})$ be an enumeration of the vertices of $R_m(G)$ such that $P_{\mu}^{R_m(G)}$ has chromatic number $k$.  
Let $P_{\mu}^{R_m(G)}=Y_{(i_1,v_{j_1})},Y_{(i_2,v_{j_2})},\ldots, Y_{(i_{m(n+1)}, v_{j_{m(n+1)}})}$.

For each $v \in V(G)$, let $t(v)$ be the vertex in $[[m],\{v\}]$ that appears first in $\mu$. 
By renaming the vertices in $G$, we may assume that $t(v_j)$ comes before $t(v_{j'})$ in $\mu$ for all $1\le j <j' \le n$.
Let $\sigma=v_1,v_2,\ldots,v_n$ be the corresponding enumeration of $V(G)$. We claim that $\sigma$ is a special enumeration of $V(G)$. 
For each $v \in V(G)$, let $f(v)\in [m]$ be such that $t(v)=(f(v),v)$.
By (1) of Lemma~\ref{bound},  $P_{\sigma}^G$ has chromatic number at most $k$. Hence, $\chi(P_{\sigma}^G)=k$.
Let $P_{\sigma}^G=X_1,X_2,\ldots,X_n$. 

Suppose $\sigma$ is not special. 
Then, there exists $\ell \in [n]$ such that $\chi(X_{\ell})=k$ and $v_{\ell}$ has a neighbor in $\{v_1,v_2,\ldots,v_{\ell-1}\}$.
Let $I_0 =\{i \mid (i,v_0) \in \mu( \leq t(v_{\ell}))\}$. By (2) of Lemma~\ref{bound}, $|I_0| \le k$. Let $I=[m] \setminus (I_0 \cup \{f(v_1),\ldots,f(v_{\ell})\})$.
Since $|I|\ge m-k -\ell \ge n-\ell+2 > |X_{\ell}| \geq \chi(X_{\ell})$,  it follows that  $\chi ( [I,X_{\ell}])=k$ and for every $k$-coloring of $\tuple{[I,X_{\ell}]}$, $[\{i\},X_{\ell}]$ sees all colors  for every $i \in I$ by Lemma \ref{coloring}. Let $(i,v)$ be the first vertex of $[I,X_{\ell} \cup \{v_0\}]$ that appears in $\mu$. Either $v=v_0$ or $(i,v)$ is adjacent to $(i,v_0)$.  In either case, $Y_{(i,v)}$ contains $[I,X_{\ell}] \cup \{(i,v_0)\}$. Since  $P_{\mu}^{R_m(G)}$ has chromatic number $k$, there exists a $k$-coloring $C$ of $\tuple{Y_{(i,v)}}$. Note that $C[[\{i\},X_{\ell}]]=[k]$. But $(i,v_0)$ is complete to $[\{i\},X_{\ell}]$, so there is no available color for $(i,v_0)$, a contradiction.  
\end{subproof}

Let $\sigma=v_1,v_2,\ldots,v_n$ be an enumeration of $V(G)$ with $\chi(P_{\sigma}^G)=k$. 
Let $\mu$ be the following enumeration of $V(R_m(G))$
\[
(1,v_1),\dots,(m,v_1),\dots,(1,v_n), \dots, (m,v_n),(1,v_0),\dots,(m,v_0).
\]
Let $P_{\mu}^{R_m(G)}=Y_{(1,v_1)},Y_{(2,v_1)},\dots,Y_{(m,v_1)},\ldots,Y_{(m,v_n)},Y_{(1,v_0)},\ldots,Y_{(m,v_0)}$.

\begin{clm} \label{mainclaim}
For all $i \in [m]$ and all $j \in [n]$, the chromatic number of $\tuple{Y_{(i,v_j)}}$ is at most $\chi(X_j)+1$.
\end{clm}

\begin{subproof}[Subproof]
Suppose $\chi(X_j)=c$ and let $(U_1,U_2,\ldots,U_c)$ be a partition of $X_j$ into independent sets of $G$.
Observe that $Y_{(i,v_j)}$ is a subset of $[[m],X_j\cup \{v_0\}]$, and 
$$
[[m],X_j\cup \{v_0\}] = [[m],\{v_0\}] \cup \left(\bigcup_{p=1}^c [[m],U_p]\right).
$$
Each set in the union is independent in $R_m(G)$, thus it follows that $\chi(Y_{(i,v_j)}) \leq c+1$.
\end{subproof}

\begin{clm}
The chromatic number of $P_{\mu}^{R_m(G)}$ is at most $k+1$. 
\end{clm}

\begin{subproof}[Subproof]
For every $i \in [m]$, $Y_{(i,v_0)}$ is a subset of $[[m],\{v_0\}]$ which is an independent set of $R_m(G)$. Hence, $\chi(Y_{(i,v_0)})\le 1$. By Claim~\ref{mainclaim}, the chromatic number of $\tuple{Y_{(i,v_j)}}$ is at most $k+1$ for $i\in [m], j\in [n]$. Thus, $\chi(P_{\mu}^{R_m(G)}) \le k+1$, as required.
\end{subproof}

\begin{clm}
If $\sigma$ is special, then $P_{\mu}^{R_m(G)}$ has chromatic number $k$.
\end{clm}

\begin{subproof}[Subproof]
Fix $i^* \in [m]$ and $j^* \in \{0\} \cup [n]$.  We will show that $\chi(Y_{({i}^*,v_{j^{*}})}) \le k$. If $j^*=0$, then $\chi(Y_{({i}^*,v_{j^{*}})}) =1$, so may assume $j^*\neq 0$.   
By Claim~\ref{mainclaim}, if $\chi(X_{j^{*}})\le k-1$, then $\chi(Y_{({i}^*,v_{j^{*}})}) \le k$. Hence, we may assume that $\chi(X_{j^{*}})=k$ and that $v_{j^{*}}$ has no neighbor in $\{v_1,v_2,\ldots,v_{j^*-1}\}$ by the definition of a special enumeration.

By Lemma~\ref{special}, there is a partition $(U_1^*,U_2^*,\ldots,U_{k-1}^*)$ of $X_{j^{*}} \setminus v_{j^{*}}$ into independent sets of $G$. For $i>i^*$, ${(i,v_{j^{*}})}$ has no neighbor in $\mu(<(i^*,v_{j^{*}}))$ since $v_{j^{*}}$ has no neighbor in $\{v_1,v_2,\ldots,v_{j^*-1}\}$.
So, $Y_{(i^*,v_{j^{*}})}$ is contained in $[[m],X_{j^{*}}\setminus v_{j^{*}}\cup \{v_0\}] \cup \{(i^*,v_{j^{*}})\}$. 

Let $C$ be the map from  $Y_{(i^*,v_{j^{*}})}$ to $[k]$ defined as
\begin{itemize}
\item for $i \neq i^*$, $C((i,v))=s$ for all $v \in U_s^*$ and $C((i,v_0))=k$,
\item $C((i^*,v))=k$ for all $v \in X_{j^{*}}$, and
\item $C((i^*,v_0))=k-1$.
\end{itemize}
It is easy to see that $C$ is a $k$-coloring of $\tuple{ Y_{(i,v_{j^{*}})}}$. 
Thus, $P_{\mu}^{R_m(G)}$ has chromatic number $k$, as required.
\end{subproof}
This last claim completes the entire proof.
\end{proof}

We finish the paper  by proving Theorem \ref{mainsec}.

\begin{reptheorem}{mainsec}
Let $G$ be a graph with $\chi_P(G)=k$ and $|V(G)|=n$. For all integers $\ell$ and $m$ such that $m \ge n+k+2$ and $\ell \ge m (n+1)+k+3$, the path-chromatic number of $R_{\ell}(R_{m}(G))$ is strictly larger than $k$.
\end{reptheorem}

\begin{proof}

Since $m\ge n+k+2$, Theorem~\ref{mainthm} shows that $R_{m}(G)$ has chromatic number either $k$ or $k+1$.

If $\chi_P(R_{m}(G)) =k+1$, then since $\ell \ge m(n+1)+k+3 = |V(R_{m}(G))| +(k+1)+2$, the path chromatic number of $R_{\ell}(R_{m}(G))$ is either $k+1$ or $k+2$ which is strictly bigger than $k$. So, we may assume that $\chi_P(R_{m}(G)) = k$.

To prove that $\chi_P(R_{\ell}(R_{m}(G))) > \chi_P(R_{m}(G))$, it suffices to show that there is no special vertex enumeration of $R_{m}(G)$ by Theorem~\ref{mainthm}.
Towards a contradiction, let $\mu=(i_1,v_{j_1}),(i_2,v_{j_2}),\ldots , (i_{m(n+1)}, v_{j_{m(n+1)}})$ be a special vertex enumeration of $R_{m}(G)$.
Let $P_{\mu}^{R_{m}(G)} =Y_{(i_1,v_{j_1})},Y_{(i_2,v_{j_2})},\ldots, Y_{(i_{m(n+1)}, v_{j_{m(n+1)}})}$.

For each vertex $v_j$ of $G$, let $t(v_j)$ be the vertex that appears first in $\mu$ among $[[m],\{v_j\}]$. 
We may assume that $t(v_j)$ comes before $t(v_{j'})$ in $\mu$ for every $1\le j < j' \le n$. 
Let $f(v_j) \in [m]$ be such that $t(v_j)=(f(v_j),v_j)$.  Let $\sigma=v_1,\ldots,v_n$. 

By (1) of Lemma~\ref{bound}, $P_{\sigma}^G$ has chromatic number $k$. 
Choose $j \in [n]$ such that $\chi( X_j)=k$. We claim that $\chi(X_j \setminus v_j)=k-1$.
Let $I_0=\{i \in [m] \mid (i,v_0)\in \mu(<t(v_j))\}$ and $I=[m] \setminus (I_0 \cup \{f(v_1),f(v_2),\ldots, f(v_j)\})$. 
By (2) of Lemma~\ref{bound}, $|I_0|\le k$. 
So, $|I|\ge m - k -j \ge (n-j+1)+1 > |X_j \setminus v_j| \ge \chi(X_j \setminus v_j)$. By Lemma~\ref{isomorphic}, 
$$
\chi([I,X_j \setminus v_j]) \ge \chi(X_j \setminus v_j ).
$$
Note that $Y_{t(v_j)}\setminus t(v_j)$ contains $[I,X_j\setminus v_j]$. 
So, if $\chi(Y_{t(v_j)}) <k$ then $\chi([I,X_j\setminus v_j]) <k$, and if $\chi(Y_{t(v_j)}) =k$ then by Lemma~\ref{special},  $\chi([I,X_j\setminus v_j])<k$ as well. In either case, 
$$
k-1 \ge \chi([I,X_j \setminus v_j]).
$$
Combining these inequalities, we obtain $k-1 \ge \chi(X_j \setminus v_j )$. 
Moreover, it is obvious that  $\chi(X_j \setminus v_j ) \ge k-1$ since $\chi(X_j) =k$. 
Therefore, $\chi(X_j \setminus v_j ) = k-1$.

Again, as $|I|> \chi(X_j \setminus v_j)$, it follows that for every $(k-1)$-coloring of $\tuple{[I,X_j\setminus v_j]}$, and every $i \in I$, $[\{i\}, X_j\setminus v_j]$ sees all $k-1$ colors by Lemma~\ref{coloring}.

Let $(i,v)$ be the first vertex of  $[I,X_j \cup \{v_0\}] $ that appears in $\mu$. 
Observe that $Y_{(i,v)}$ contains $[I,X_j \setminus v_j] \cup \{(i,v_0)\}$. For every $(k-1)$-coloring of $\tuple{[I,X_j \setminus v_j]}$, $[\{i\},X_j\setminus v_j]$ sees all $k-1$ colors, so $\tuple{[I,X_j \setminus v_j] \cup \{(i,v_0)\}}$ is not $(k-1)$-colorable since $(i,v_0)$ is complete to $[\{i\},X_j\setminus v_j]$. Thus, $\chi([I,X_j \setminus v_j] \cup \{(i,v_0)\})=k$, and so $\chi(Y_{(i,v)})=k$.
Since $\mu$ is special, $(i,v)$ has no neighbor in $\mu(<(i,v))$. So, $v$ is either $v_0$ or $v_j$.

By Lemma~\ref{special}, $\tuple{Y_{(i,v)}} \setminus (i,v)$ is $(k-1)$-colorable. 
If $v=v_0$ then $Y_{(i,v)} \setminus (i,v)$ contains $[I,X_j]$. By Lemma~\ref{isomorphic}, $\tuple{[I,X_j]}$ contains $\tuple{X_j}$ as an induced subgraph, contradicting that $\tuple{Y_{(i,v)}\setminus (i,v)}$ is $(k-1)$-colorable.
If $v=v_j$ then $Y_{(i,v)} \setminus (i,v)$ contains $[I,X_j \setminus v_j] \cup \{(i,v_0)\}$. 
Again, the chromatic number of $\tuple{[I,X_j \setminus v_j] \cup \{(i,v_0)\}}$ is $k$, a contradiction.

Therefore, $\mu$ is not special. This completes the proof.
\end{proof}

\textbf{Acknowledgments.} We would like to thank Jan-Oliver Fröhlich, Irene Muzi, Claudiu Perta and Paul Wollan for many helpful discussions. We would also like to thank the anonymous referees for numerous suggestions in improving the paper.

\bibliography{references}{}
\bibliographystyle{plain}
\end{document}